\documentclass[12pt]{article}
\usepackage{mathtools}
\usepackage{amsmath,amssymb,amsbsy}
\usepackage{graphicx}
\usepackage{multirow}
\usepackage{tikz}
\usepackage{tkz-graph}
\usetikzlibrary{shapes.geometric}
\usepackage{amsfonts}
\usepackage{amsthm}
\usepackage{amssymb}
\usepackage{graphicx, enumerate}
\usepackage{amsmath}
\usepackage{latexsym}
\usepackage{longtable}
\usepackage{tabularx}
\usepackage{amsmath}
\usepackage{amsfonts}
\usepackage{amsthm}
\usepackage{setspace}
\usepackage{graphicx}
\usepackage{float}
\usepackage{rotating}
\usepackage{tikz}
\usepackage{verbatim}
\usepackage{soul}

\def\zet{\mathbb{Z}}

\def\fk2{\lfloor\frac{k}{2}\rfloor}
\def\ck2{\lceil\frac{k}{2}\rceil}
\def\k1{\lfloor\frac{k+1}{2}\rfloor}

\newtheorem{theorem}{Theorem}[section]

\newtheorem{observation}[theorem]{Observation}
\newtheorem{corollary}[theorem]{Corollary}
\newtheorem{conjecture}[theorem]{Conjecture}

\newtheorem{problem}[theorem]{Problem}

\title{$E_A$-cordial labeling of graphs and its implications for $A$-antimagic labeling of trees}
\author{Sylwia Cichacz\\\\
\normalsize AGH University of Krak\'{o}w, \vspace{2mm} Poland}

\begin{document}

\maketitle
\begin{abstract}
If $A$ is a finite Abelian group, then a labeling $f \colon E (G) \rightarrow A$ of the edges of some graph $G$ induces a vertex labeling on $G$; the vertex $u$ receives the label $\sum_{v\in N(u)}f (v)$, where $N(u)$ is an open neighborhood of the vertex $u$. A graph $G$ is $E_A$-cordial if there is an edge-labeling such that (1) the edge label classes
differ in size by at most one and (2) the induced vertex label classes differ in size by at most one. Such a labeling is called $E_A$-cordial. In the literature, so far only $E_A$-cordial labeling in cyclic groups has been studied. 

The corresponding problem was studied by Kaplan, Lev and Roditty. Namely, they introduced $A^*$-antimagic labeling as a generalization of antimagic labeling \cite{ref_KapLevRod}. 
Simply saying, for a tree of order $|A|$ the $A^*$-antimagic labeling is such $E_A$-cordial labeling that the label $0$ is prohibited on the edges. 

In this paper, we give necessary and sufficient conditions for paths to be $E_A$-cordial for any cyclic $A$. We also show that the conjecture for $A^*$-antimagic labeling of trees posted in \cite{ref_KapLevRod} is not true.
\end{abstract}


\section{Introduction}
Assume $A$ is a finite Abelian group of order $n$ with the operation denoted by $+$.  For convenience we will write $ka$ to denote $a + a + \cdots + a$ (where the element $a$ appears $k$ times), $-a$ to denote the inverse of $a$, and use $a - b$ instead of $a+(-b)$.  Moreover, the notation $\sum_{a\in S}{a}$ will be used as a short form for $a_1+a_2+a_3+\cdots+a_t$, where $a_1, a_2, a_3, \dots,a_t$ are all elements of the set $S$. The identity element of $A$ will be denoted by $0$. Recall that any group element $\iota\in\Gamma$ of order 2 (i.e., $\iota\neq 0$ and $2\iota=0$) is called an \emph{involution}. Let us denote the number of involutions in $\Gamma$ by $|I(\Gamma)|$.

For a graph $G=(V,E)$,  an Abelian group $A$ and an $A$-labeling
$c: V \to A$ let $v_c(a)=|c^{-1}(a)|$.
 The labeling $c$ induces an edge
labeling $c^*:E \to A$  defined by $c^*(e)=\sum_{v \in
e}c(v)$, let $e_{c}(a)=|{c^*}^{-1}(a)|$. The labeling $c$ is called an $A$-\textit{cordial labeling} if
$|v_c(a)-v_c(b)| \leq 1$ and $|e_{c}(a)-e_{c}(b)|\leq 1$ for any
$a,b \in A$.
 A graph $G$ is said to be {\it $A$-cordial} if it admits an $A$-cordial
labeling $c$.

The notion of $\zet_2$-cordial labeling of graphs was introduced by Cahit \cite{Cahit} as a weakened
version of graceful labeling and harmonious labeling. This notion was generalized by Hovey for any Abelian group \cite{Hovey}. 
Hovey proved that cycles are $\zet_k$-cordial for all odd $k$; for $k$ even $C_{2mk+j}$ is $\zet_k$-cordial when $0\leq j \leq \frac{k}{2}+ 2$ and when $k < j <2k$. Moreover he showed that $C_{(2m+1)k}$ is not $\zet_k$-cordial and posed a conjecture that  for $k$ even the cycle $C_{2mk + j}$ where $0 \leq j < 2k$, is $\zet_k$-cordial if and
only if $j \neq k$. This conjecture was verified by Tao \cite{Tao}.  Tao’s result combined with those of Hovey show that:
\begin{theorem}[\cite{Hovey,Tao}]\label{cycle}
The cycle $C_n$ is $\zet_k$-cordial if and only if $k$ is odd or  $n$ is
not an odd multiple of $k$.
\end{theorem}
It was proved recently.
\begin{theorem}[\cite{Cic}]\label{Cic}
Any cycle $C_n$, $n>2$ is $A$-cordial, if $|A|$ is odd.
Any path $P_n$ is $A$-cordial unless $A\cong (\zet_2)^m$ for $m>1$ and  $n=|A|$ or $n=|A|+1$.
\end{theorem}

Yilmaz and Cahit introduced an edge version of cordial labeling called $E_k$-cordial. For a graph $G=(V,E)$,  a labeling
$f: E \to \zet_k$, $k>1$ let $e_f(a)=|f^{-1}(a)|$.
 The labeling $f$ induces a vertex
labeling $f^*:V \to \zet_k$  defined by $f^*(v)=\sum_{u \in
N(v)}f(u)$, let $v_{f}(a)=|{f^*}^{-1}(a)|$. The labeling $f$ is called an $E_k$-\textit{cordial labeling} if
$|v_f(a)-v_f(b)| \leq 1$ and $|e_{f}(a)-e_{f}(b)|\leq 1$ for any
$a,b \in \zet_k$. If a graph has an $E_k$-cordial labeling, then we call it $E_k$-cordial graph.

Yilmaz and Cahit proved the following:
\begin{theorem}[\cite{YilCah}]\label{Yil}
    If a graph $G$ has $n\equiv 2\pmod 4$ vertices, then it is not $E_2$-cordial.  
\end{theorem}
They also showed that this is a sufficient condition for trees and cycles to be $E_2$-cordial. Liu, Liu, and Wu considered $E_k$-cordial labeling for $k>2$ and proved the following.
\begin{theorem}[\cite{LiuLiu}]\label{Liu}
    A path $P_n$, $n\ge3$ has $E_p$-cordial labeling for odd $p$.  
\end{theorem}
We generalize the notion of $E_k$-cordial labeling for any finite Abelian group $A$. Namely, for a graph $G=(V,E)$,  an Abelian group $A$ and an $A$-labeling
$f: E \to A$ let $e_f(a)=|f^{-1}(a)|$.
 The labeling $f$ induces a vertex
labeling $f^*:V \to A$  defined by $f^*(v)=\sum_{u \in
N(v)}f(u)$, let $v_{f}(a)=|{f^*}^{-1}(a)|$. The labeling $f$ is called an $E_A$-\textit{cordial labeling} if
$|v_f(a)-v_f(b)| \leq 1$ and $|e_{f}(a)-e_{f}(b)|\leq 1$ for any
$a,b \in A$.
 A graph $G$ is said to be {\it $E_A$-cordial} if it admits an $E_A$-cordial
labeling $f$.

One of the most well-known conjectures in the field of graph labeling is the {\em antimagic conjecture} by Hartsfield and Ringel~\cite{HarRin}, which claims that the edges of every graph except $K_2$ can be labeled by integers $1,2,\dots,|E|$ so that the weight of each vertex is different.
In \cite{ref_KapLevRod}, Kaplan, Lev and Roditty considered the following generalization of the
concept of an anti-magic graph. For a finite Abelian group, let $A^* = A \setminus \{0\}$.
An {$A^*$-antimagic labeling} of
$G$ is a bijection from the set of edges to
the set $A^*$, such that all the vertex sums are pairwise
distinct. 
They posted the following conjecture:
\begin{conjecture}[\cite{ref_KapLevRod}]\label{false}
    A tree with $|A|$ vertices is $A^*$-antimagic if and only if $|I(A)|\neq 1$.
\end{conjecture}
It was proved  $4$-tree\footnote{A $4$-tree $T$ is a rooted tree, where every vertex that is not a leaf has at least four children.} $T$ of order $n$ admits an $A^*$-antimagic labeling  for any group $A$ of order $n$ such that $|I(\Gamma)|\neq 1$ \cite{CS}. However, Kaplan et al. \cite{ref_KapLevRod} showed that, if $A$ has a unique involution, then every tree on $n$ vertices is not $A^*$-antimagic, the conjecture is not true. It is enough to consider a path $P_{2^m}$, $m>1$ and a group $A\cong(\zet_2)^m$. Since all non-zero elements in $A$ are involutions, there are not two of them that sum up to $0$. Note, that paths are not the only example of such trees, others are trees of order $n$ with maximum degree $n-2$. On the other hand, there exist  trees with vertices of degree $2$ possessing $((\zet_2)^m)^*$-antimagic labeling, an example of such tree  is shown in Figure~\ref{tree}.

 \begin{figure}[ht!]
\begin{center}
\begin{tikzpicture}[scale=0.7,style=thick,x=1cm,y=1cm]
\def\vr{3pt} 

\path (6,6) coordinate (v0);
\path (0,3) coordinate (v1);
\path (4,3) coordinate (v2);
\path (8,3) coordinate (v3);
\path (12,3) coordinate (v4);
\path (0,0) coordinate (v5);
\path (8,0) coordinate (v6);
\path (12,0) coordinate (v7);

\draw (v0) -- (v1) -- (v5);
\draw (v0) -- (v2);
\draw (v0) -- (v3);
\draw (v0) -- (v4) -- (v7);
\draw (v3) -- (v6);

\draw (0.75,4.15) node {$(1,1,1)$};
\draw (3.75,4.15) node {$(1,0,1)$};
\draw (8.2,4.15) node {$(0,1,1)$};
\draw (11.2,4.15) node {$(0,0,1)$};
\draw (-1.15,1.15) node {$(1,0,0)$};
\draw (6.85,1.15) node {$(0,1,0)$};
\draw (13.15,1.15) node {$(1,1,0)$};

\draw (v0) [fill=black] circle (\vr);
\draw (v1) [fill=black] circle (\vr);
\draw (v2) [fill=black] circle (\vr);
\draw (v3) [fill=black] circle (\vr);
\draw (v4) [fill=black] circle (\vr);
\draw (v5) [fill=black] circle (\vr);
\draw (v6) [fill=black] circle (\vr);
\draw (v7) [fill=black] circle (\vr);

\draw[anchor = south] (v0) node {$(0,0,0)$};
\draw[anchor = east] (v1) node {$(0,1,1)$};
\draw[anchor = north] (v2) node {$(1,0,1)$};
\draw[anchor = west] (v3) node {$(0,0,1)$};
\draw[anchor = west] (v4) node {$(1,1,1)$};
\draw[anchor = north] (v5) node {$(1,0,0)$};
\draw[anchor = north] (v6) node {$(0,1,0)$};
\draw[anchor = north] (v7) node {$(1,1,0)$};
\end{tikzpicture}
\end{center}

\caption{A $((\zet_2)^3)^*$-antimagic labeling for a tree.}\label{tree}
\end{figure}
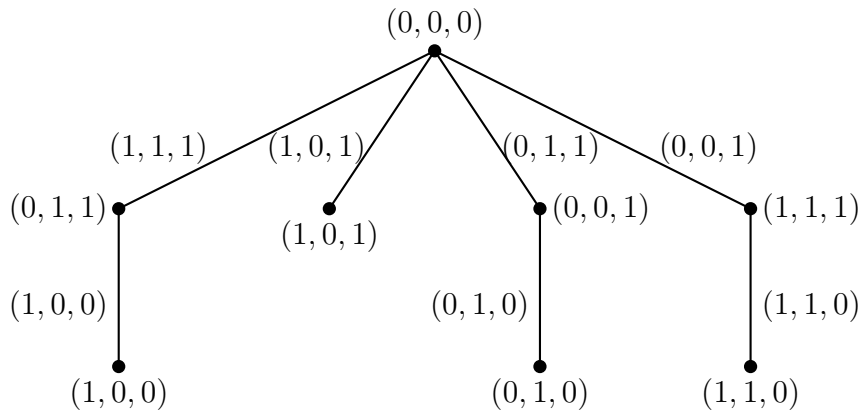

Note that an $A^*$-antimagic labeling of a tree is $E_A$-cordial labeling, but the inverse is not true. Taking the above considerations into account, we will weaken the assumptions of the definition $A^*$-antimagic labeling by letting the identity element $0$ be a label of an edge and we call such labeling $A$-antimagic. Thus a tree of order $|A|$ is  $E_A$-cordial if and only if it is $A$-cordial.

In this paper, we will generalize Theorem~\ref{Liu} by 
giving necessary and sufficient conditions for a path to being $E_k$-cordial for any $k$. Moreover, we show that any cycle $C_n$ and any path $P_n$ is $E_A$-cordial for any $A$ of odd order.   Finally, we show that a path $P_n$ is $A$-antimagic if and only if $n\not\equiv2\pmod 4$. 
\section{Preliminaries}\label{subsec:preliminaries}

The fundamental theorem of finite Abelian groups states that every finite Abelian group $\Gamma$ is isomorphic to the direct product of some cyclic subgroups of prime-power orders. In other words, there exists a positive integer $k$, (not necessarily distinct) prime numbers $\{p_i\}_{i=1}^{k}$, and positive integers $\{\alpha_i\}_{i=1}^{k}$, such that
$$A\cong\zet_{p_1^{\alpha_1}}\oplus\zet_{p_2^{\alpha_2}}\oplus\ldots\oplus\zet_{p_k^{\alpha_k}} \mathrm{, with}\; m = p_1^{\alpha_1}\cdot p_2^{\alpha_2}\cdot\ldots\cdot p_k^{\alpha_k},$$ where $m$ is the order of $A$. Moreover, this group factorization is unique (up to the order of terms in the direct product). {Since any cyclic finite group of even order has exactly one involution, if $e$ is the number of cyclic subgroups in the factorization of $\Gamma$ whose order is even, then $|I(A)|=2^e-1$.}

Because the results presented in this paper are invariant under the isomorphism between groups ($\cong$), we only need to consider one group in every isomorphism class. 

Recall that, for a prime number $p$, a group the order of which is a power of $p$ is called a \textit{$p$-group}. Given a finite Abelian group $A$, \textit{Sylow $p$-subgroup} of $A$ is the maximal subgroup $L$ of $A$ the order of which is a power of $p$. For example, by the fundamental theorem of finite Abelian groups, it is easy to see that any finite Abelian group $A$ can be factorized as $A\cong L\times H$, where $L$ is the Sylow $2$-group of $A$ and the order of $H$ is odd.

\section{Main results}
We start with simple observations.
\begin{observation}\label{cc}
A cycle $C_n$ is $A$-cordial if and only if it is $E_A$-cordial.
\end{observation}
\begin{proof}
Let $C_n=v_1e_1v_2\ldots v_ne_nv_1$, where $v_i\in V(C_n)$ and $e_j\in E(C_n)$. 
Let $c$ be $A$-cordial labeling of $C_n$. Define $f(e_i)=c(v_i)$ for $i=1,2,\ldots,n$. One can easily check that $f$ is $E_A$-cordial labeling of $C_n$.  The inverse implication can be obtained in a similar way.\end{proof}

\begin{observation}\label{cp}
If a cycle $C_n$ is $E_A$-cordial, then a path $P_n$ is also $E_A$-cordial.
\end{observation}
\begin{proof}
 Let $f$ be an $E_A$-cordial labeling $f$ of $C_n$. Note that $e_{f}(a)\in\left\{\left\lfloor \frac{n}{|A|}\right\rfloor,\left\lceil\frac{n}{|A|}\right\rceil\right\}$. 
        Let $g\in A$ be such that $e_{f}(g)=\left\lceil\frac{n}{|A|}\right\rceil$. Define $f'\colon A\to E(C_n)$ such that $f'(e)=f(e)-g$. Observe that $f'$ is  an $E_A$-cordial labeling such that $e_{f'}(0)=\left\lceil\frac{n}{|A|}\right\rceil$. By deleting now an edge $e$ such that $f'(e)=0$, we obtain a path $P_n$ with $E_A$-cordial labeling.      
\end{proof}

Observation~\ref{cc} and Theorem~\ref{Cic} imply immediately the following:

\begin{corollary}\label{cycle2}
    Any cycle $C_n$  is $E_A$-cordial for any $A$ of odd order.~\qed
 \end{corollary}

 By above Corollary~\ref{cycle2} and Observation~\ref{cp} we obtain the following. 

\begin{corollary}\label{oddpaths}
    Any path $P_n$  is $E_A$-cordial for any $A$ of odd order.~\qed
\end{corollary}

We will focus now on groups with even order. 

\begin{theorem}\label{ant}
    Let $A\cong \zet_{4m}\oplus H$ for $m>1$ and $|H|=k\geq 1$ odd, $n=|A|$. There exists an $E_A$-cordial labeling of $P_n$.
\end{theorem}
\begin{proof}

  By Observation~\ref{cc} and Corollary~\ref{cycle2} there exists an $E_H$-cordial labeling $f_1$ of a cycle $C_k=w_0w_1\ldots w_{k-1}w_0$. Without loss of generality, we can assume that $f_1(w_0w_1)=0$.

Break the path $P_{4mk}$ into $k$ blocks of $4m$ vertices.  Namely, let
$$P_{4mk}=v_{0,0},v_{0,1}\ldots,v_{0,4m-1},v_{1,0},v_{1,1}\ldots,v_{1,4m-1},\ldots v_{k-1,0},v_{k-1,1},\ldots,v_{k-1,4m-1}.$$

Set $$f(v_{j,4m-1}v_{j+1,0})=\begin{cases}(0,f_1(w_{j+1}w_{j+2}))& \mathrm{for}\;\; j\;\;\mathrm{even},\\
(2m,f_1(w_{j+1}w_{j+2}))& \mathrm{for}\;\; j\;\;\mathrm{odd}.\end{cases}$$

Let $$f(v_{0,i}v_{0,i+1})=\begin{cases}(i/2,f_1(w_0w_1))& \mathrm{for}\;\; i=0,2,\ldots,4m-2\\
(2m+(i-1)/2,f_1(w_0w_1))& \mathrm{for}\;\; i=1,3,\ldots,2m-1,\\
(2m+1+(i-1)/2,f_1(w_0w_1))& \mathrm{for}\;\; i=2m+1,2m+3,\ldots,4m-3,\\
\end{cases}$$
Observe that  
$$f^*(v_{0,i})=\begin{cases}
(0,0)&\mathrm{for}\; i=0,\\
(2m-1+i,0)&\mathrm{for}\; i=1,2,\ldots,2m,\\
(2m+i,0)&\mathrm{for}\; i=2m+1,2m+2,\ldots,4m-2.\end{cases}$$

For $j$ odd set
$$f(v_{j,i}v_{j,i+1})=\begin{cases}(2m+i/2,f_1(w_{j}w_{j+1}))& \mathrm{for}\;\; i=0,2,\ldots,4m-2\\
(1+(i-1)/2,f_1(w_jw_{j+1}))& \mathrm{for}\;\; i=1,3,\ldots,4m-3.\\
\end{cases}$$
Whereas for $j>0$ even let
$$f(v_{j,i}v_{j,i+1})=\begin{cases}(i/2,f_1(w_{j}w_{j+1}))& \mathrm{for}\;\; i=0,2,\ldots,4m-2\\
(2m+1+(i-1)/2,f_1(w_jw_{j+1}))& \mathrm{for}\;\; i=1,3,\ldots,4m-3.\\
\end{cases}$$
Observe that    
$f^*(v_{j,i})=(2m+i,2f_1(w_{j}w_{j+1}))$ for $i=0,1,\ldots,4m-2$ and $j=1,2,\ldots,k-1$ and $f^*(v_{j,4m-1})=(2m-1,f_1(w_{j}w_{j+1}+f_1(w_{j+1}w_{j+2}))$ for $j=0,1,\ldots,k-1$, where the subscripts are taken modulo $k$.\end{proof}

In  Figure~\ref{c12} we show an $E_A$-cordial labeling for $C_{12}$, where $A=\zet_4\oplus\zet_3$, $m=2$, $k=3$, whereas in Figure we show an $E_A$-cordial labeling for $C_{12}$, where $A=\zet_{24}$, $m=6$, $k=1$.

 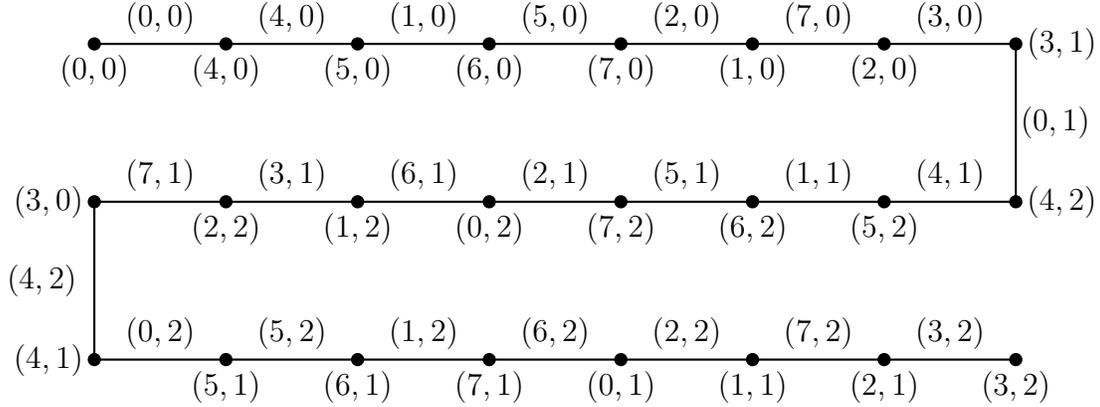
\begin{figure}[ht!]
\begin{center}
\begin{tikzpicture}[scale=0.7,style=thick,x=1cm,y=1cm]
\def\vr{3pt} 

\path (0,6) coordinate (v0);
\path (2.5,6) coordinate (v1);
\path (5,6) coordinate (v2);
\path (7.5,6) coordinate (v3);
\path (10,6) coordinate (v4);
\path (12.5,6) coordinate (v5);
\path (15,6) coordinate (v6);
\path (17.5,6) coordinate (v7);
\path (0,3) coordinate (v8);
\path (2.5,3) coordinate (v9);
\path (5,3) coordinate (v10);
\path (7.5,3) coordinate (v11);
\path (10,3) coordinate (v12);
\path (12.5,3) coordinate (v13);
\path (15,3) coordinate (v14);
\path (17.5,3) coordinate (v15);
\path (0,0) coordinate (v16);
\path (2.5,0) coordinate (v17);
\path (5,0) coordinate (v18);
\path (7.5,0) coordinate (v19);
\path (10,0) coordinate (v20);
\path (12.5,0) coordinate (v21);
\path (15,0) coordinate (v22);
\path (17.5,0) coordinate (v23);

\draw (v0) -- (v1) -- (v2) -- (v3) -- (v4) -- (v5) -- (v6) -- (v7);

\draw (v8) -- (v9) -- (v10) -- (v11) -- (v12) -- (v13) -- (v14) -- (v15);
\draw  (v16) -- (v17) -- (v18) -- (v19) -- (v20) -- (v21) -- (v22) -- (v23);

 \draw (v8) -- (v16);

 \draw (v7) -- (v15);

\draw (1.25,6.5) node {$(0,0)$};
\draw (3.75,6.5) node {$(4,0)$};
\draw (6.25,6.5) node {$(1,0)$};
\draw (8.75,6.5) node {$(5,0)$};
\draw (11.25,6.5) node {$(2,0)$};
\draw (13.75,6.5) node {$(7,0)$};
\draw (16.25,6.5) node {$(3,0)$};
\draw (18.25,4.5) node {$(0,1)$};
\draw (1.25,3.5) node {$(7,1)$};
\draw (3.75,3.5) node {$(3,1)$};
\draw (6.25,3.5) node {$(6,1)$};
\draw (8.75,3.5) node {$(2,1)$};
\draw (11.25,3.5) node {$(5,1)$};
\draw (13.75,3.5) node {$(1,1)$};
\draw (16.25,3.5) node {$(4,1)$};
\draw (-1,1.5) node {$(4,2)$};
\draw (1.25,0.5) node {$(0,2)$};
\draw (3.75,0.5) node {$(5,2)$};
\draw (6.25,0.5) node {$(1,2)$};
\draw (8.75,0.5) node {$(6,2)$};
\draw (11.25,0.5) node {$(2,2)$};
\draw (13.75,0.5) node {$(7,2)$};
\draw (16.25,0.5) node {$(3,2)$};

\draw (v0) [fill=black] circle (\vr);
\draw (v1) [fill=black] circle (\vr);
\draw (v2) [fill=black] circle (\vr);
\draw (v3) [fill=black] circle (\vr);
\draw (v4) [fill=black] circle (\vr);
\draw (v5) [fill=black] circle (\vr);
\draw (v6) [fill=black] circle (\vr);
\draw (v7) [fill=black] circle (\vr);
\draw (v8) [fill=black] circle (\vr);
\draw (v9) [fill=black] circle (\vr);
\draw (v10) [fill=black] circle (\vr);
\draw (v11) [fill=black] circle (\vr);
\draw (v12) [fill=black] circle (\vr);
\draw (v13) [fill=black] circle (\vr);
\draw (v14) [fill=black] circle (\vr);
\draw (v15) [fill=black] circle (\vr);
\draw (v16) [fill=black] circle (\vr);
\draw (v17) [fill=black] circle (\vr);
\draw (v18) [fill=black] circle (\vr);
\draw (v19) [fill=black] circle (\vr);
\draw (v20) [fill=black] circle (\vr);
\draw (v21) [fill=black] circle (\vr);
\draw (v22) [fill=black] circle (\vr);
\draw (v23) [fill=black] circle (\vr);

\draw[anchor = north] (v0) node {$(0,0)$};
\draw[anchor = north] (v1) node {$(4,0)$};
\draw[anchor = north] (v2) node {$(5,0)$};
\draw[anchor = north] (v3) node {$(6,0)$};
\draw[anchor = north] (v4) node {$(7,0)$};
\draw[anchor = north] (v5) node {$(1,0)$};
\draw[anchor = north] (v6) node {$(2,0)$};
\draw[anchor = west] (v7) node {$(3,1)$};
\draw[anchor = west] (v15) node {$(4,2)$};
\draw[anchor = north] (v14) node {$(5,2)$};
\draw[anchor = north] (v13) node {$(6,2)$};
\draw[anchor = north] (v12) node {$(7,2)$};
\draw[anchor = north] (v11) node {$(0,2)$};
\draw[anchor = north] (v10) node {$(1,2)$};
\draw[anchor = north] (v9) node {$(2,2)$};
\draw[anchor = east] (v8) node {$(3,0)$};
\draw[anchor = east] (v16) node {$(4,1)$};
\draw[anchor = north] (v17) node {$(5,1)$};
\draw[anchor = north] (v18) node {$(6,1)$};
\draw[anchor = north] (v19) node {$(7,1)$};
\draw[anchor = north] (v20) node {$(0,1)$};
\draw[anchor = north] (v21) node {$(1,1)$};
\draw[anchor = north] (v22) node {$(2,1)$};
\draw[anchor = north] (v23) node {$(3,2)$};
\end{tikzpicture}
\end{center}
\caption{An $E_{\zet_8\oplus\zet_3}$-cordial labeling for $P_{24}$.}\label{c12}
\end{figure}

 \begin{figure}[ht!]
\begin{center}
\begin{tikzpicture}[scale=0.7,style=thick,x=1cm,y=1cm]
\def\vr{3pt} 

\path (0,6) coordinate (v0);
\path (2.5,6) coordinate (v1);
\path (5,6) coordinate (v2);
\path (7.5,6) coordinate (v3);
\path (10,6) coordinate (v4);
\path (12.5,6) coordinate (v5);
\path (15,6) coordinate (v6);
\path (17.5,6) coordinate (v7);
\path (0,3) coordinate (v8);
\path (2.5,3) coordinate (v9);
\path (5,3) coordinate (v10);
\path (7.5,3) coordinate (v11);
\path (10,3) coordinate (v12);
\path (12.5,3) coordinate (v13);
\path (15,3) coordinate (v14);
\path (17.5,3) coordinate (v15);
\path (0,0) coordinate (v16);
\path (2.5,0) coordinate (v17);
\path (5,0) coordinate (v18);
\path (7.5,0) coordinate (v19);
\path (10,0) coordinate (v20);
\path (12.5,0) coordinate (v21);
\path (15,0) coordinate (v22);
\path (17.5,0) coordinate (v23);

\draw (v0) -- (v1) -- (v2) -- (v3) -- (v4) -- (v5) -- (v6) -- (v7);

\draw (v8) -- (v9) -- (v10) -- (v11) -- (v12) -- (v13) -- (v14) -- (v15);
\draw  (v16) -- (v17) -- (v18) -- (v19) -- (v20) -- (v21) -- (v22) -- (v23);

 \draw (v8) -- (v16);

 \draw (v7) -- (v15);

\draw (1.25,6.5) node {$0$};
\draw (3.75,6.5) node {$12$};
\draw (6.25,6.5) node {$1$};
\draw (8.75,6.5) node {$13$};
\draw (11.25,6.5) node {$2$};
\draw (13.75,6.5) node {$14$};
\draw (16.25,6.5) node {$3$};
\draw (18.25,4.5) node {$15$};
\draw (1.25,3.5) node {$7$};
\draw (3.75,3.5) node {$19$};
\draw (6.25,3.5) node {$6$};
\draw (8.75,3.5) node {$17$};
\draw (11.25,3.5) node {$5$};
\draw (13.75,3.5) node {$16$};
\draw (16.25,3.5) node {$4$};
\draw (-1,1.5) node {$20$};
\draw (1.25,0.5) node {$8$};
\draw (3.75,0.5) node {$21$};
\draw (6.25,0.5) node {$9$};
\draw (8.75,0.5) node {$22$};
\draw (11.25,0.5) node {$10$};
\draw (13.75,0.5) node {$23$};
\draw (16.25,0.5) node {$11$};

\draw (v0) [fill=black] circle (\vr);
\draw (v1) [fill=black] circle (\vr);
\draw (v2) [fill=black] circle (\vr);
\draw (v3) [fill=black] circle (\vr);
\draw (v4) [fill=black] circle (\vr);
\draw (v5) [fill=black] circle (\vr);
\draw (v6) [fill=black] circle (\vr);
\draw (v7) [fill=black] circle (\vr);
\draw (v8) [fill=black] circle (\vr);
\draw (v9) [fill=black] circle (\vr);
\draw (v10) [fill=black] circle (\vr);
\draw (v11) [fill=black] circle (\vr);
\draw (v12) [fill=black] circle (\vr);
\draw (v13) [fill=black] circle (\vr);
\draw (v14) [fill=black] circle (\vr);
\draw (v15) [fill=black] circle (\vr);
\draw (v16) [fill=black] circle (\vr);
\draw (v17) [fill=black] circle (\vr);
\draw (v18) [fill=black] circle (\vr);
\draw (v19) [fill=black] circle (\vr);
\draw (v20) [fill=black] circle (\vr);
\draw (v21) [fill=black] circle (\vr);
\draw (v22) [fill=black] circle (\vr);
\draw (v23) [fill=black] circle (\vr);

\draw[anchor = north] (v0) node {$0$};
\draw[anchor = north] (v1) node {$12$};
\draw[anchor = north] (v2) node {$13$};
\draw[anchor = north] (v3) node {$14$};
\draw[anchor = north] (v4) node {$15$};
\draw[anchor = north] (v5) node {$16$};
\draw[anchor = north] (v6) node {$17$};
\draw[anchor = west] (v7) node {$18$};
\draw[anchor = west] (v15) node {$19$};
\draw[anchor = north] (v14) node {$20$};
\draw[anchor = north] (v13) node {$21$};
\draw[anchor = north] (v12) node {$22$};
\draw[anchor = north] (v11) node {$23$};
\draw[anchor = north] (v10) node {$1$};
\draw[anchor = north] (v9) node {$2$};
\draw[anchor = east] (v8) node {$3$};
\draw[anchor = east] (v16) node {$4$};
\draw[anchor = north] (v17) node {$5$};
\draw[anchor = north] (v18) node {$6$};
\draw[anchor = north] (v19) node {$7$};
\draw[anchor = north] (v20) node {$8$};
\draw[anchor = north] (v21) node {$9$};
\draw[anchor = north] (v22) node {$10$};
\draw[anchor = north] (v23) node {$11$};
\end{tikzpicture}
\end{center}

\caption{An $E_{\zet_{24}}$-cordial labeling for $P_{24}$.}\label{c12}
\end{figure}
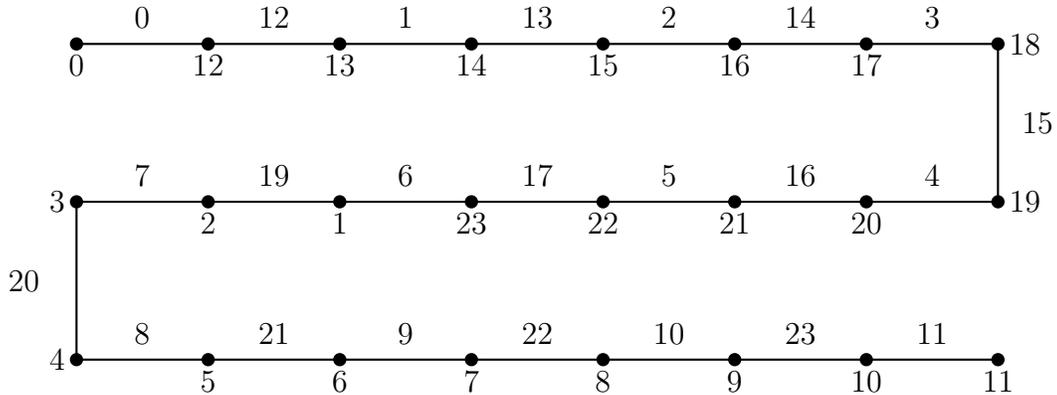

\begin{observation}\label{product}
Let $A=B\oplus H$ be of order $n$. If a tree $T$ of order $n$ is $E_A$-cordial, then it is $E_B$-cordial.
\end{observation}
\begin{proof}
Suppose that $f$ is $E_H$-cordial labeling of $T$. If now $f(e)=(b,h)$ for $b\in B$ and $h\in H$, then define a labeling $f'\colon E(T)\to B$ by $f'(e)=b$. Obviously $|v_{f'}(a)-v_{f'}(b)| =0$, since $v_{f}(a)=1$ for any $a\in A$. Moreover, because there exists exactly one $a\in A$ such that $e_{f}(a)=0$  we obtain that $|e_{f'}(a)-e_{f'}(b)|\leq 1$ for any
$a,b \in B$.
\end{proof}

\begin{observation}\label{2mod4}
 If a tree $T$ has $n\equiv 2\pmod 4$ vertices, then it is not $E_A$-cordial for any $A$ of order $|A|\equiv2\pmod 4$.
\end{observation}
\begin{proof}
By fundamental theorem on abstract algebra $A\cong \zet_2\oplus H$, where $|H|$ is odd. Therefore, if there exists an $E_A$-cordial-labeling of $G$, then by Observation~\ref{product} there exists a $E_2$-cordial labeling of $G$, which is impossible by Theorem~\ref{Yil}.
\end{proof}

We will show now the necessary and sufficient conditions for a path $P_n$ to be $E_k$-cordial.

\begin{theorem}\label{cyclic}  A path $P_n$ is $E_k$-cordial if and only if $k\not\equiv2\pmod 4$  or  $n$ is
not an odd multiple of $k$.
\end{theorem}
\begin{proof}
If $k$ is odd or $n$ is not an odd multiple of $k$ then we are done by Theorem~\ref{cycle}, Observation~\ref{cc} and Observation~\ref{cp}. Assume now that $k\equiv 0\pmod 4$ and $n$ is
 an odd multiple of $k$. Thus $n=4m$ for some odd $m$.  Suppose first that $k=4$. If  also $n=4$, then for a path $P_4=v_1v_2v_3v_4$, the labeling defined as $f(v_1v_2)=0$, $f(v_2v_3)=1$ and $f(v_3v_4)=2$ is $E_4$-cordial. Assume now that $n>4$.
  Let $A=\zet_4\oplus \zet_m\cong \zet_{4m}$. There exists $E_{A}$-cordial labeling of $P_n$ by Theorem~\ref{ant}. Thus Observation~\ref{product} implies that there exists $E_4$-cordial labeling of $P_n$.
  Let now $k>4$, then let $A=\zet_k\oplus \zet_m$ and we are done as above.
  
  If now $k\equiv2\pmod 4$  and  $n$ is
 an odd multiple of $k$, then  $n\equiv2\pmod 4$ and we are done by Observation~ \ref{2mod4}. \end{proof}

\section{$A$-antimagic labeling of paths}

A useful tool for us is a concept of sum-rainbow Hamiltonian cycles on Abelian groups \cite{Lev}. Consider a finite Abelian group $A$ and the complete graph on the set of all elements of
$A$.  In this graph find a Hamiltonian cycle and calculate the sum for each pair of consecutive vertices along the cycle. The goal is to determine the maximum number of distinct sums, denoted as $\sigma_{\max}(A)$, that can arise from this process. 
\begin{theorem}[\cite{Lev}]\label{Lev}
For any finite non-trivial Abelian group $A$ we have
$$\sigma_{\max}(A)=\begin{cases}
|A|& if\; A\not \cong (\zet_2)^m,\; |I(A)|\neq 1;\\
|A|-1& if \; |I(A)|=1;\\
|A|-2& if \; A\cong (\zet_2)^m, \;  |I(A)|\neq 1;
\end{cases}$$
\end{theorem}

This problem is also connected with the concept of \textit{$R$-sequenceability} of groups. A group $A$ of order $n$ is said to be $R$-\textit{sequenceable} if the nonidentity elements of the group can be listed in a sequence $g_1, g_2,\ldots, g_{n-1}$ such that  $g_1^{-1}g_2, g_2^{-1}g_3,\ldots, g_{n-1}^{-1}g_1$ are all distinct. This concept was introduced in  1974 by Ringel \cite{Ringel}, who used this concept in his solution of Heawood map coloring problem.
 An Abelian group is $R^*$-sequenceable if it has an $R$-sequencing $g_1, g_2,\ldots, g_{n-1}$ such that $g_{i-1}g_{i+1}=g_i$ for some $i$ (subscripts are read modulo $n - 1$). The term was introduced by Friedlander et al. \cite{Friedlander}, who showed that the existence of an $R^*$-sequenceable Sylow 2-subgroup is a sufficient condition for a group to be $R$-sequenceable. The following theorem has been proved:
\begin{theorem}[\cite{Headley}]\label{Headley} An Abelian group whose Sylow 2-subgroup is noncyclic and not of order 8 is $R^*$-sequenceable.
\end{theorem}

\begin{theorem}
    A path $P_n$ is $A$-antimagic if and only if $n\not \equiv 2 \pmod 4$.
\end{theorem}
\begin{proof}
  Assume first that $|I(A)|=1$, if now $n \equiv 2 \pmod 4$, then we are done by Theorem~\ref{2mod4}. For $n \equiv 0 \pmod 4$ there is $A\cong \zet_4$ or by fundamental theorem on abstract algebra $A\cong\zet_{4m}\oplus H$ for $|H|$ odd and $m\geq 1$.
  If $A\cong \zet_4$ then we apply Theorem~\ref{cyclic}, otherwise Theorem~\ref{ant}.
Suppose now that $|I(A)|\neq 1$, if now $A\not\cong(\zet_2)^m$, $m>1$ (i.e $|I(A)|\neq1$, then by Observation~\ref{cc} and Theorem~\ref{Lev} there exists an $E_A$-cordial labeling  of $C_n$ and hence there exists an $E_A$-cordial labeling of $P_n$ by Observation~\ref{cp}. Recall that this means that $P_n$ is $A$-antimagic.
From now $A\cong(\zet_2)^m$ for some $m>1$.
If $m=3$, then the desired labeling is given in Figure~\ref{p8}.

 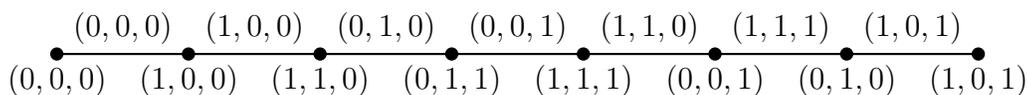
\begin{figure}[ht!]
\begin{center}
\begin{tikzpicture}[scale=0.7,style=thick,x=1cm,y=1cm]
\def\vr{3pt} 

\path (0,0) coordinate (v16);
\path (2.5,0) coordinate (v17);
\path (5,0) coordinate (v18);
\path (7.5,0) coordinate (v19);
\path (10,0) coordinate (v20);
\path (12.5,0) coordinate (v21);
\path (15,0) coordinate (v22);
\path (17.5,0) coordinate (v23);

\draw  (v16) -- (v17) -- (v18) -- (v19) -- (v20) -- (v21) -- (v22) -- (v23);

\draw (1.25,0.5) node {$(0,0,0)$};
\draw (3.75,0.5) node {$(1,0,0)$};
\draw (6.25,0.5) node {$(0,1,0)$};
\draw (8.75,0.5) node {$(0,0,1)$};
\draw (11.25,0.5) node {$(1,1,0)$};
\draw (13.75,0.5) node {$(1,1,1)$};
\draw (16.25,0.5) node {$(1,0,1)$};

\draw (v16) [fill=black] circle (\vr);
\draw (v17) [fill=black] circle (\vr);
\draw (v18) [fill=black] circle (\vr);
\draw (v19) [fill=black] circle (\vr);
\draw (v20) [fill=black] circle (\vr);
\draw (v21) [fill=black] circle (\vr);
\draw (v22) [fill=black] circle (\vr);
\draw (v23) [fill=black] circle (\vr);

\draw[anchor = north] (v16) node {$(0,0,0)$};
\draw[anchor = north] (v17) node {$(1,0,0)$};
\draw[anchor = north] (v18) node {$(1,1,0)$};
\draw[anchor = north] (v19) node {$(0,1,1)$};
\draw[anchor = north] (v20) node {$(1,1,1)$};
\draw[anchor = north] (v21) node {$(0,0,1)$};
\draw[anchor = north] (v22) node {$(0,1,0)$};
\draw[anchor = north] (v23) node {$(1,0,1)$};
\end{tikzpicture}
\end{center}
\caption{A ${(\zet_{2})^3}$-antimagic labeling for $P_{8}$.}\label{p8}
\end{figure}

  Let $n=2^m$ for $m\neq 3$ and $a_1,a_2,\ldots,a_{n-1}$ be an $R^*$-sequence of $A$ which exists by Theorem~\ref{Headley}. Note that since $a_i=-a_i$ we obtain that
 the sequence $a_1+a_2,a_2+a_3,\ldots,a_{n-1}+a_1$ is injective. Without loss of generality, we can assume that $a_2=a_1+a_{n-1}$. Note that this implies $a_{n-1}=a_1+a_2$. Let $P_n=v_1v_2\ldots v_n$. Define an $A$-antimagic  labeling as $f(v_1v_2)=0$ and $f(v_iv_{i+1})=a_i$ for $i=2,3,\ldots,n-1$.
\end{proof}
\section{Conclusions}
In this paper, we gave sufficient and necessary conditions for paths $P_n$ possessing an $A$-antimagic labeling. By Observation~\ref{2mod4} no tree of order $n\equiv2\pmod 4$ admits an $A$-cordial labeling for any $A$. Therefore, we state the following conjecture:
\begin{conjecture}
    A tree $T$ is $A$-antimagic if and only if $|T|\not\equiv2\pmod 4$.
\end{conjecture}

Taking into account Conjecture~\ref{false}, we believe it is worth posing the following open problem.
\begin{problem}
    Characterize trees such that they have $A^*$-antimagic labeling for any $A$ with $|I(A)|\neq1$.
\end{problem}

\section{Statements and Declarations}
This work was  supported by program ''Excellence initiative – research university'' for the AGH University.


\end{document}